\documentclass[10pt]{amsart}

\usepackage{amssymb,amsmath,amsthm}
\usepackage{graphicx,a4wide}

\theoremstyle{plain}

\usepackage{graphicx,tikz}
\newtheorem{theorem}{Theorem}

\newtheorem*{corollary}{Corollary}

\newtheorem*{lemma}{Lemma}

\theoremstyle{definition}

\theoremstyle{remark}

\newcommand{\diam}{\operatorname{diam}}

\begin{document}

\title[]{An Endpoint Alexandrov Bakelman \\Pucci estimate in the plane}


\author[]{Stefan Steinerberger}
\address{Department of Mathematics, Yale University}
\email{stefan.steinerberger@yale.edu}

\begin{abstract} The classical Alexandrov-Bakelman-Pucci estimate for the Laplacian states 
$$ \max_{x \in \Omega}{ |u(x)|} \leq \max_{x \in \partial \Omega}{|u(x)|} + c_{s,n} \diam(\Omega)^{2-\frac{n}{s}} \left\| \Delta u\right\|_{L^s(\Omega)}$$
where $\Omega \subset \mathbb{R}^n$, $u \in C^2(\Omega) \cap C(\overline{\Omega})$ and $s > n/2$. The inequality fails for $s = n/2$. A Sobolev embedding result of Milman \& Pustylink, originally phrased in a slightly different context, implies an endpoint inequality: if $n \geq 3$ and $\Omega \subset \mathbb{R}^n$ is bounded, then
$$ \max_{x \in \Omega}{ |u(x)|} \leq \max_{x \in \partial \Omega}{|u(x)|} + c_n \left\| \Delta u\right\|_{L^{\frac{n}{2},1}(\Omega)},$$
where $L^{p,q}$ is the Lorentz space refinement of $L^p$. This inequality fails for $n=2$ and we prove a sharp substitute result: there exists $c>0$ such that for all $\Omega \subset \mathbb{R}^2$ with finite measure
$$ \max_{x \in \Omega}{ |u(x)|} \leq \max_{x \in \partial \Omega}{|u(x)|} + c \max_{x \in \Omega} \int_{y \in \Omega}{   \max\left\{ 1, \log{\left(\frac{|\Omega|}{\|x-y\|^2}  \right)} \right\} \left| \Delta u(y)\right| dy}.$$
This is somewhat dual to the classical Trudinger-Moser inequality -- we also note that it is sharper than the usual estimates given in Orlicz spaces, the proof is rearrangement-free.
The Laplacian can be replaced by any uniformly elliptic operator in divergence form.
\end{abstract}

\maketitle

\vspace{-5pt}

\section{Introduction and main results}
\subsection{Introduction.}
The Alexandrov-Bakelman-Pucci estimate \cite{alex1, alex2, bak, pucci, pucci2} is one of the classical estimates in the study of elliptic partial differential equations. 
In its usual form it is stated for a second order uniformly elliptic operator 
$$ Lu = a_{ij}(x) \partial_{ij} u + b_i(x) \partial_i u$$
with bounded measurable coefficients in a bounded domain $\Omega \subset \mathbb{R}^n$ and $c(x) \leq 0$. The Alexandrov-Bakelman-Pucci estimate then states that for
any $u \in C^2(\Omega) \cap C(\overline{\Omega})$
$$ \sup_{x \in \Omega}{ |u(x)|} \leq \sup_{x \in \partial \Omega}{|u(x)|} + c \diam(\Omega) \left\| L u\right\|_{L^{n}(\Omega)},$$
where $c$ depends on the ellipticity constants of $L$ and the $L^n-$norms of the $b_i$. It is a rather foundational maximum principle and discussed in most of the standard
textbooks, e.g.  Caffarelli \& Cabr\'{e} \cite{xavi}, Gilbarg \& Trudinger \cite{gilbarg}, Han \& Lin \cite{han} and Jost \cite{jost}. The ABP estimate has inspired a very active field of research, we do not
attempt a summary and refer to \cite{cabre2, cabre3,xavi,  gilbarg, tso} and references therein.
Alexandrov \cite{ad2} and Pucci \cite{pucci2} showed that $L^n$ can generically not be replaced by a smaller norm. 
However, for some elliptic operators operators it is possible to get estimates with $L^p$ with $p < n$, see \cite{astala}.
We will start our discussion with the special case of the Laplacian, where the inequality reads, for any $s > n/2$, 
$$ \max_{x \in \Omega}{ |u(x)|} \leq \max_{x \in \partial \Omega}{|u(x)|} + c_{s,n} \diam(\Omega)^{2-\frac{n}{s}} \left\| \Delta u\right\|_{L^s(\Omega)}.$$

\subsection{Results.} The inequality is known to fail in the endpoint $s=n/2$. 
The purpose of our short paper is to note endpoint versions of the inequality. The first result is essentially due to
Milman \& Pustylink \cite{milman} (see also \cite{milman2}), with an alternative proof due to
 Xiao \& Zhai \cite{xiao} (although ascribing it to anyone in particular is not an easy matter, one could reasonably argue that Talenti's seminal paper \cite[Eq. 20]{talenti} already
contains the result without spelling it out).

\begin{theorem} Let $n \geq 3$, let $\Omega \subset \mathbb{R}^n$ be bounded and $u \in C^2(\Omega) \cap C(\overline{\Omega})$. Then 
$$ \max_{x \in \Omega}{ |u(x)|} \leq \max_{x \in \partial \Omega}{|u(x)|} + c_n \left\| \Delta u\right\|_{L^{\frac{n}{2},1}(\Omega)},$$
where $c_n$ only depends on the dimension.
\end{theorem}

Here $L^{n/2,1}$ is the Lorentz space refinement of $L^{n/2}$. We note that its norm is slightly larger than $L^{n/2}$ and this turns
out to be sufficient to establish an endpoint result in a critical space for which the geometry of $\Omega$ now longer enters into the inequality. We refer to
Grafakos \cite{grafakos} for an introduction to Lorentz spaces. The proofs given in \cite{milman, milman2, milman3, talenti} rely on rearrangement techniques.
Theorem 1 fails for $n=2$: the Lorentz spaces collapse to $L^{1,1} = L^1$ and the inequality is false in $L^1$  (see below). We obtain
a sharp endpoint result in $\mathbb{R}^2$.

\begin{theorem}[Main result] Let $\Omega \subset \mathbb{R}^2$ have finite measure and let $u \in C^2(\Omega) \cap C(\overline{\Omega})$. Then
$$ \max_{x \in \Omega}{ |u(x)|} \leq \max_{x \in \partial \Omega}{|u(x)|} + c \max_{x \in \Omega} \int_{y \in \Omega}{   \max\left\{ 1, \log{\left(\frac{|\Omega|}{\|x-y\|^2}  \right)} \right\} \left| \Delta u(y)\right| dy}.$$
\end{theorem}
The result seems to be new. We observe that Talenti \cite{talenti} is hinting at the proof of a slightly weaker result using rearrangement techniques (after his equation (22), see a recent
paper of Milman \cite{milman3} for a complete proof and related results).
$\Omega$ need not be bounded, it suffices to assume that it has finite measure. 
We illustrate sharpness of the inequality with an example on the unit disk: define the radial function $u_{\varepsilon}(r)$ by
$$ u(r) = \begin{cases} \frac{1}{2} - \log{\varepsilon} -  \frac{1}{2}\varepsilon^{-2} r^2 \qquad &\mbox{if}~0 \leq r \leq \varepsilon\\
- \log{r} \qquad &\mbox{if}~\varepsilon \leq r \leq 1. \\ \end{cases}$$
 We observe that $\Delta u_{\varepsilon} \sim \varepsilon^{-2} 1_{\left\{|x| \leq \varepsilon\right\}} $ and $\|u\|_{L^{\infty}} \sim \log{(1/\varepsilon)}$. This shows that the
solution is unbounded as $\varepsilon \rightarrow 0$ while $\|\Delta u\|_{L^1} \sim 1$ remains bounded; in particular, no Alexandrov-Bakelman-Pucci inequality in 
$L^{1}$ is possible for $n=2$.
The example also shows
Theorem 2 to be sharp: the maximum is assumed in the origin and
$$ \int_{y \in \Omega}{   \max\left\{ 1, \log{\left(\frac{|\Omega|}{\|y\|^2}  \right)} \right\}    \varepsilon^{-2} 1_{\left\{|y| \leq \varepsilon\right\}} dy} = \frac{1}{\varepsilon^2} \int_{B(0,\varepsilon)}{ \log{ \left( \frac{\pi}{\|y\|^2}\right)}dy} \sim \log{\left( \frac{1}{\varepsilon}\right)}.$$
The proof will show that the constant $|\Omega|$ inside the logarithm is quite natural but can be improved if the domain is very different from a disk: indeed,
we can get sharper result that recover some of the information that is lost in applying rearrangement type techniques and with a slight modification of the main argument we can
obtain a slightly stronger result capturing more geometric information.

\begin{corollary} Let $\Omega \subset \mathbb{R}^2$ have finite measure and be simply connected and let $u \in C^2(\Omega) \cap C(\overline{\Omega})$. Then
$$ \max_{x \in \Omega}{ |u(x)|} \leq \max_{x \in \partial \Omega}{|u(x)|} + c \max_{x \in \Omega} \int_{y \in \Omega}{   \max\left\{ 1, \log{\left(\frac{\emph{inrad}(\Omega)^2}{\|x-y\|^2}  \right)} \right\} \left| \Delta u(y)\right| dy}.$$
\end{corollary}
All results remain true if we replace the Laplacian $-\Delta$ by a uniformly elliptic operator in divergence form $-\mbox{div}(a(x)\cdot \nabla u)$ or replace $\mathbb{R}^n$ by a manifold as long as the induced heat kernel satisfies Aronson-type bounds \cite{aronson}.

\subsection{Related results.} There is a trivial connection between Alexandrov-Bakelman-Pucci estimates and second-order Sobolev inequalities that, to the best of our knowledge, has never been
made explicit. After constructing 
\begin{align*}
 \Delta \phi &= 0 \qquad \mbox{in}~\Omega\\ 
\phi &= u \qquad \mbox{on}~\partial \Omega
\end{align*}
we may trivially estimate, using the maximum principle for harmonic functions,
$$ \max_{x \in \Omega}{ |u(x)|} \leq \max_{x \in  \Omega}{ |\phi(x)|} + \max_{x \in \Omega}{ |u(x) - \phi(x)|}  \leq \max_{x \in \partial \Omega}{ |u(x)|} + \max_{x \in \Omega}{ |u(x) - \phi(x)|}.$$
This reduces the problem to studying functions $u \in C^2(\Omega)$ that vanish on the boundary and verifying the validity of estimates of the type
$$ \| u\|_{L^{\infty}(\Omega)} \lesssim_{\Omega} \| \Delta u \|_{X}.$$
The Alexandroff-Bakelman-Pucci estimate is one such estimate. These objects have been actively studied for a long time, see e.g. \cite{cianchi, biez2, xiao} and references therein. 
 Theorem 1 can thus be restated as second-order Sobolev inequality in the endpoint $p=\infty$ and requiring a Lorentz-space refinement; it can be equivalently stated as
$$ \| u \|_{L^{\infty}(\mathbb{R}^n)} \leq c_n \| \Delta u \|_{L^{\frac{n}{2}, 1}(\mathbb{R}^n)} \qquad \mbox{for all}~u \in C^{\infty}_c(\mathbb{R}^n), ~n \geq 3.$$
This inequality seems to have first been stated in the literature by Milman \& Pustylink \cite{milman} in the context of Sobolev embedding at the critical scale. Xiao \& Zhai \cite{xiao} derive the inequality via harmonic analysis. The failure of the embedding of the critical Sobolev space into $L^{\infty}$ is classical
$$ W_{0}^{2, \frac{n}{2}}(\Omega) \not\hookrightarrow L^{\infty}(\Omega).$$
There are two natural options: one could either try to find a slightly larger space $Y \supset L^{\infty}(\Omega)$
to have a valid embedding or one could try to find a space slightly smaller than the Sobolev space to have a valid embedding. The result of Milman \& Pustylink \cite{milman}  deals with the second question. From the point of view of studying Sobolev spaces, the first question is quite a bit more relevant since it investigates extremal behavior of
functions in a Sobolev space and has been addressed in many papers \cite{adams, bastero, brezis, moser, milman, perez}. We emphasize the Trudinger-Moser inequality \cite{moser,trudinger}: for $\Omega \subset \mathbb{R}^2$
$$ \sup_{\|\nabla u\|_{L^2} \leq 1}{ \int_{\Omega}{ e^{4\pi |u|^2} dx}} \leq c |\Omega|.$$
Cassani, Ruf \& Tarsi \cite{cas} prove a variant: the condition $\| \Delta u\|_{L^1} < \infty$ suffices to ensure that $u$ has at most logarithmic blow-up. 
These results should be seen as somewhat dual to Theorem 2. Put differently, 
Theorem 2 is a natural converse to this result since it implies that any function with $\| \Delta u \|_{L^1} < \infty$ and logarithmic blow-up has a Laplacian
$\Delta u$ that concentrates its $L^1-$mass.

\section{Proofs}
The proofs are all based on the idea of representing a function $u:\Omega \rightarrow \mathbb{R}$ as the stationary solution of the heat equation with a suitably chosen
right-hand side (these techniques have recently proven useful in a variety of problems \cite{biswas, lierl, manas, stein})
\begin{align*}
v_t + \Delta v &= \Delta u \qquad \mbox{in}~\Omega \\
v &= u \qquad \mbox{on}~\partial \Omega.
\end{align*}
The Feynman-Kac formula then implies a representation of $u(x) = v(t,x)$ as a convolution of the heat kernel and its values in a neighborhood to which
standard estimates can be applied. We use $\omega_x(t)$ to denote
Brownian motion started in $x \in \Omega$ at time $t$; moreover, in accordance with Dirichlet boundary conditions, we will assume that
the boundary is sticky and remains at the boundary once it touches it. The Feynman-Kac formula then implies that for all $t > 0$
$$ u(x) = \mathbb{E} u(\omega_x(t)) + \mathbb{E} \int_{0}^{t}{ (\Delta u)(\omega_x(t)) dt}.$$
This representation will be used in all our proofs. The proof of Theorem 1 will be closely related in spirit to \cite[Lemma 3.2.]{xiao} phrased in
a different language; this language turns out to be useful in the proof of Theorem 2 where an additional geometric argument is required.

\subsection{A Technical Lemma.} The purpose of this section is to quickly prove a fairly basic inequality. The Lemma
already appeared in a slightly more precise form in work of Lierl and the author \cite{lierl}. We only need a special case and prove it for 
completeness of exposition.
\begin{lemma} Let $n \in \mathbb{N}$, let $t>0$, $c_1, c_2 > 0$ and $0 \neq x \in \mathbb{R}^n$. We have
$$ \int_{0}^{t}{ \frac{c_1}{s} \exp \left( -\frac{\|x\|^2}{c_2 s} \right)}ds \lesssim_{c_1, c_2} \left(1 + \max\left\{0, -\log{\left( \frac{\|x\|^2}{c_2 t} \right)} \right\} \right) \exp\left(-\frac{\|x\|^2}{c_2 t}\right).$$
and, for $n \geq 3$, 
$$ \int_{0}^{\infty}{ \frac{c_1}{s^{n/2}} \exp \left( -\frac{\|x\|^2}{c_2 s} \right)}ds \lesssim_{c_1, c_2,n }  \frac{1}{\|x\|^{n-2}}.$$
\end{lemma}

\begin{proof} The substitutions $z= s/|x|^2$ and $y=1/(c_2 z)$ show
$$  \int_{0}^{t}{ \frac{c_1}{s^{}} \exp \left( -\frac{|x|^2}{c_2 s} \right)}ds \lesssim_{c_1,c_2}   \int_{|x|^2/(c_2 t)}^{\infty}{y^{-1} e^{-y} dy}.$$
If $|x|^2/(c_2 d) \leq 1$ we have
 $$ \int_{|x|^2/(c_2 t)}^{\infty}{y^{-1} e^{-y} dy} \lesssim 1 + \int_{|x|^2/(c_2 t)}^{1}{y^{-1} e^{-y} dy} \lesssim 1  + \int_{|x|^2/(c_2 t)}^{1}{y^{-1} dy} \lesssim 1 - \log{ \left(\frac{|x|^2}{c_2 t} \right)}, $$
and if $|x|^2/(c_2 t) \geq 1$ we have
 $$ \int_{|x|^2/(c_2 t)}^{\infty}{y^{-1} e^{-y} dy} \leq \frac{c_2 d}{|x|^2} \int_{|x|^2/(c_2 t)}^{\infty}{e^{-y} dy} =  \frac{c_2 t}{|x|^2} \exp\left(-\frac{|x|^2}{c_2 t}\right) \leq \exp\left(-\frac{|x|^2}{c_2 t}\right).$$ 
Summarizing, this establishes
 $$ \int_{|x|^2/(c_2 t)}^{\infty}{\frac{1}{y} e^{-y} dy} \lesssim \left(1 + \max\left\{0, -\log{\left( \frac{|x|^2}{c_2 t} \right)} \right\} \right) \exp\left(-\frac{|x|^2}{c_2 t}\right),$$
which is the desired statement for $n=2$. The second statement, for $n \geq 3$, is trivial.
\end{proof}

\subsection{Proof of Theorem 1}

\begin{proof} 
We rewrite $u$ as the stationary solution of the heat equation
\begin{align*}
v_t + \Delta v &= \Delta u \qquad \mbox{in}~\Omega \\
v &= u \qquad \mbox{on}~\partial \Omega.
\end{align*}
As explained above, the Feynman-Kac formula implies that for all $t > 0$
$$ u(x) = v(t,x) = \mathbb{E} v(\omega_x(t)) + \mathbb{E} \int_{0}^{t}{ (\Delta u)(\omega_x(t)) dt}.$$
Let $x$ be arbitrary, we now let $t \rightarrow \infty$. The first term is quite simple since we recover the harmonic measure. Indeed, as $ t \rightarrow \infty$, we have
$$ \lim_{t \rightarrow \infty}\mathbb{E} v(\omega_x(t)) = \phi(x) \qquad \mbox{where} \quad \begin{cases} \Delta \phi = 0 ~ \mbox{inside}~\Omega \\
\phi = u ~\mbox{on}~\partial \Omega. \end{cases}$$
This can be easily seen from the stochastic interpretation of harmonic measure. This implies 
$$ \lim_{t \rightarrow \infty}\mathbb{E} v(\omega_x(t))  \leq \max_{x \in \partial \Omega}{ u(x)}.$$

It remains to estimate the second term. We denote the heat kernel on $\Omega$ by $p_{\Omega}(t,x,y)$ and observe 
\begin{align*}
 \left| \mathbb{E} \int_{0}^{t}{(\Delta u)(\omega_x(t)) dt} \right| &\leq \mathbb{E} \int_{0}^{t}{ \left| \Delta u(\omega_x(t))\right| dt} \\
&= \int_{0}^{t}{  \int_{y \in \Omega}{ p_{\Omega}(s,x,y) \left| \Delta u(y)\right| dy} ds} \\
&\leq \int_{y \in \Omega}{  \left( \int_{0}^{\infty}{p_{\Omega}(s,x,y)ds} \right)   \left| \Delta u(y)\right|  dy}
\end{align*}
However, using domain monotonicity $ p_{\Omega}(t,x,y) \leq p_{\mathbb{R}^n}(t, x, y)$ as well as
 the explicit Gaussian form of the heat kernel on $\mathbb{R}^n$ and the Lemma we have,  uniformly in $x,y \in \Omega$,
$$   \int_{0}^{\infty}{p_{\Omega}(s,x,y)ds} \leq  \int_{0}^{\infty}{p_{\mathbb{R}^n}(s,x,y)ds} \leq \frac{c_n}{\|x - y\|^{n-2}}.$$
The duality of Lorentz spaces 
 $$ \| fg \|_{L^1(\mathbb{R}^n)} \leq \|f\|_{L^{\frac{n}{2}, 1}(\mathbb{R}^n)} \| g\|_{L^{\frac{n}{n-2}, \infty}(\mathbb{R}^n)} \quad \mbox{and} \quad \frac{1}{\|x-y\|^{n-2}} \in L^{\frac{n}{n-2},\infty}(\mathbb{R}^n,dy)$$
then implies the desired result
$$ \left| \mathbb{E} \int_{0}^{t}{(\Delta u)(\omega_x(t)) dt} \right|   \leq  c_n\int_{y \in \Omega}{  \frac{\left| \Delta u\right|(y)}{\|x - y\|^{n-2}} dy} \leq   \left\|  \frac{c_n}{\|x-y\|^{n-2}} \right\|_{ L^{\frac{n}{n-2},\infty}}  \| \Delta u\|_{L^{\frac{n}{2}, 1}}.$$
\end{proof}

\subsection{Proof of Theorem 2}

\begin{proof} This argument requires a simple statement for Brownian motion: for all sets $\Omega \subset \mathbb{R}^2$ with finite
volume $|\Omega| < \infty$ and all  $x \in \Omega$, 
$$ \mathbb{P}\left( \exists~0 \leq t \leq \frac{|\Omega|}{8}:  w_x(t) \notin \Omega\right) \geq \frac{1}{2}.$$
We start by bounding the probability from below: for this, we introduce the free Brownian motion $\omega^*_x(t)$ that also starts in $x$ but moves freely through $\mathbb{R}^n$ without getting stuck on the boundary $\partial \Omega$. Continuity of Brownian motion then implies
$$ \mathbb{P}\left( \exists~0 \leq t \leq \frac{|\Omega|}{8}:  w_x(t) \notin \Omega\right)  \geq \mathbb{P}\left(   w^*_x\left(|\Omega|/8\right) \notin \Omega\right).$$
Moreover, we can compute
$$ \mathbb{P}\left(   w^*_x(|\Omega|/8) \notin \Omega\right) = \int_{\mathbb{R}^n \setminus \Omega}{ \frac{\exp\left(-2\|x-y\|^2/ |\Omega|\right)}{(\pi |\Omega|/2)^{}} dy}.$$
We use the Hardy-Littlewood rearrangement inequality to argue that
$$  \int_{\mathbb{R}^n \setminus \Omega}{ \frac{\exp\left(-2\|x-y\|^2/ |\Omega|\right)}{(\pi |\Omega|/2)^{}} dy} \geq  \int_{\mathbb{R}^n \setminus B}{  \frac{\exp\left(-2\|y\|^2/ |B|\right)}{(\pi |B|/2)^{}}dy},$$
where $B$ is a ball centered in the origin having the same measure as $\Omega$. However, assuming $|B| = R^2 \pi$ this quantity can be computed in polar cordinates as
$$  \int_{\mathbb{R}^n \setminus B}{  \frac{\exp\left(-2\|y\|^2/ |B|\right)}{(\pi |B|/2)^{}}dy} = \int_{R}^{\infty}{ \frac{  \exp\left(-2r^2/(R^2 \pi)\right)}{R^2 \pi^2 /2} 2 \pi r dr} = e^{-\frac{2}{\pi}} > \frac{1}{2}.$$
 We return to the representation, valid for all $t > 0$,
$$ v(t,x) = \mathbb{E} v(\omega_x(t)) + \mathbb{E}\int_{0}^{t}{ (\Delta u)(\omega_x(t)) dt}.$$
We will now work with finite values of $t$: the computation above implies that at time $t=|\Omega|$
$$   \left|  \mathbb{E} v(\omega_x(|\Omega|))\right| \leq \frac{1}{2} \max_{x \in \partial \Omega}{|u(x)|} + \frac{ \max_{x \in \Omega}{u(x)}}{2}.$$
Arguing as above and employing the Lemma shows that
\begin{align*}
 \left| \mathbb{E} \int_{0}^{|\Omega|}{(\Delta u)(\omega_x(t)) dt} \right|  &\leq \int_{y \in \Omega}{  \left( \int_{0}^{|\Omega|}{p(s,x,y)ds} \right)\left| \Delta u(y)\right| dy}  \\
&\lesssim \|\Delta u\|_{L^1} + \int_{y \in \Omega}{   \max\left\{ 0, \log{\left(\frac{|\Omega|}{\|x-y\|^2}  \right)} \right\} \left| \Delta u(y)\right| dy} \\
&\lesssim \int_{y \in \Omega}{   \max\left\{ 1, \log{\left(\frac{|\Omega|}{\|x-y\|^2}  \right)} \right\} \left| \Delta u(y)\right| dy}.
\end{align*}
We can now pick $x \in \Omega$ so that $u$ assumes its maximum there and argue
\begin{align*}
 \max_{x \in \Omega}{u(x)} &= v(|\Omega|,x) = \mathbb{E} v(\omega_x(|\Omega|)) + \mathbb{E}\int_{0}^{|\Omega|}{ (\Delta u)(\omega_x(t)) dt} \\
&\leq \frac{1}{2} \max_{x \in \partial \Omega}{|u(x)|} + \frac{ \max_{x \in \Omega}{u(x)}}{2} + c \max_{x \in \Omega}  \int_{y \in \Omega}{   \max\left\{ 1, \log{\left(\frac{|\Omega|}{\|x-y\|^2}  \right)} \right\} \left| \Delta u(y)\right| dy}
\end{align*}
which implies the desired statement.
\end{proof}

\subsection{Proof of the Corollary}
\begin{proof} 
The proof can be used almost verbatim, we only require the elementary statement that for all simply-connected domains $\Omega \subset \mathbb{R}^2$ and all $x_0\in \Omega$
$$ \mathbb{P}\left( \exists~0 \leq t \leq c \cdot \mbox{inrad}(\Omega)^2:  w_{x_0}(t) \notin \Omega   \right) \geq \frac{1}{100}. $$
The idea is actually rather simple: for any such $x_0$ there exists a point $\|x_0 - x_1\| \leq \mbox{inrad}(\Omega)$ such that $y \notin \Omega$.
 Since
$\Omega$ is simply connected, the boundary is an actual line enclosing the domain: in particular, the disk of radius $m \mbox{inrad}(\Omega)$ centered around $x_0$
either already contains the entire domain $\Omega$ or has a boundary of length at least $(2m-2) \cdot \mbox{inrad}(\Omega)$ (an example being close to the extremal case is shown
in Figure 1).
\begin{center}
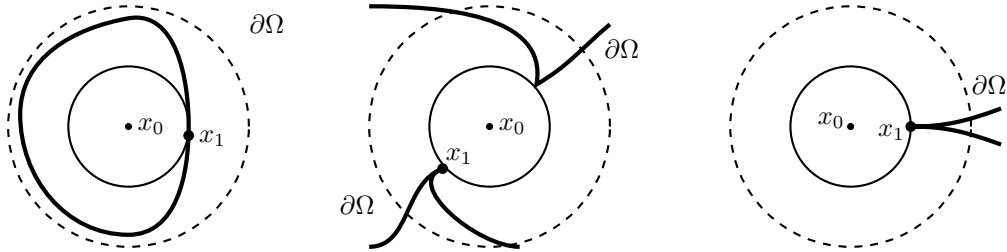
\begin{figure}[h!]
\begin{tikzpicture}[scale =0.8]
\filldraw (0,0) circle (0.05cm);
\draw [thick] (0,0) circle (1cm);
\draw [thick, dashed] (0,0) circle (2cm);
\draw [ultra thick] (-2, 2) to[out=0, in=80] (0.78, 0.7) to[out = 30, in =220] (2, 1.7);
\draw [ultra thick] (-2, -2) to[out=0, in=200] (-0.78, -0.7) to[out = 200, in =180] (0.5, -2);
\filldraw (-0.78,-0.7) circle (0.08cm);
\node at (-0.5, -0.5) {$x_1$};
\node at (0.38, 0) {$x_0$};
\node at (2.2, 1.3) {$\partial \Omega$};
\node at (-2.2, -1.3) {$\partial \Omega$};

\filldraw (-6,0) circle (0.05cm);
\draw [thick] (-6,0) circle (1cm);
\draw [thick, dashed] (-6,0) circle (2cm);
\draw [ultra thick] (-5, 0) to[out=90, in=10] (-6, 1.8) to[out = 190, in =90] (-7.8, 0.2)  to[out = 270, in =180] (-6, -1.8) to[out = 0, in =270] (-5, 0);
\filldraw (-5,-0.15) circle (0.08cm);
\node at (-4.6, -0.2) {$x_1$};
\node at (0.68+5, 0.1) {$x_0$};
\node at (0.68+6, -0.1) {$x_1$};
\filldraw (7,0) circle (0.08cm);
\node at (2.3+6, 0.7) {$\partial \Omega$};

\filldraw (6,0) circle (0.05cm);
\draw [thick] (6,0) circle (1cm);
\draw [thick, dashed] (6,0) circle (2cm);
\draw [ultra thick] (8.5, 0.3) to[out=200, in=0] (7, 0) to[out = 0, in =160] (8.5, -0.3);
\node at (0.38-6, 0) {$x_0$};
\node at (2.3-6, 1.7) {$\partial \Omega$};
\end{tikzpicture}
\caption{The point of maximum $x_0$, the circle with radius $d(x_0, \Omega)$, the circle with radius $2 d(x_0, \Omega)$ (dashed) and the possible local geometry of $\partial \Omega$.}
\end{figure}
\end{center}

It turns out that $m=2$ is already an admissible choice, the computations are carried out in earlier work of M. Rachh and the author \cite{manas}.

\end{proof}

\textbf{Acknowledgement.} The author is grateful to Mario Milman for discussions about the history of some of these results.

\end{document}